\documentclass[a4paper,11pt,reqno]{amsart}
\usepackage{amssymb,amsmath,hyperref,mathrsfs,scalerel}
\newtheorem{theorem}{Theorem}[section]
\newtheorem{lemma}[theorem]{Lemma}

\theoremstyle{definition}
\newtheorem{definition}[theorem]{Definition}
\newtheorem{remark}[theorem]{Remark}
\newtheorem{example}[theorem]{Example}
\numberwithin{equation}{section}

\newcommand{\N}{\mathbb{N}}
\newcommand{\Z}{\mathbb{Z}}
\newcommand{\Q}{\mathbb{Q}}
\newcommand{\R}{\mathbb{R}}
\newcommand{\C}{\mathbb{C}}

\newcommand{\Hsp}{\mathbb{H}}

\newcommand{\Ogroup}{\mathrm{O}}
\newcommand\tr{\mathrm{t}}
\newcommand\ssm{\smallsetminus}
\newcommand\smod[1]{\ (\operatorname{mod} #1)}
\renewcommand\theta{\vartheta}
\newcommand\ileg[2]{\bigl(\frac{#1}{#2}\bigr)}
\newcommand\leg[2]{\Bigl(\frac{#1}{#2}\Bigr)}

\title{Geometric Proofs of Riemann-Type Theta Relations}
\author{Sander Zwegers}
\date{\today}
\address{Department of Mathematics and Computer Science, University of Cologne, Weyertal 86--90, 50931 Cologne, Germany}
\email{szwegers@uni-koeln.de}

\begin{document}

\begin{abstract}
We present a geometric method for deriving theta function identities from automorphisms of positive definite quadratic forms.
Given an automorphism of a quadratic form with matrix $A$, we construct a rational isometry of $I_n\otimes A$ that fixes the diagonal subspace pointwise.
Elementary lattice decompositions and finite character sums then yield explicit transformation formulas for the associated theta functions.
As applications, we give a geometric proof of Schultz's generalization of the Borweins' cubic theta identity and recover two related quadratic theta identities.
\end{abstract}

\maketitle

\section{Introduction}

Cubic theta functions, introduced by J.M. Borwein and P.B. Borwein, play a central role in Ramanujan's theory of elliptic functions to alternative bases.
Recently, Chan, Chan, Liu, and Zudilin \cite{chan2025} revisited a multivariable generalization due to Schultz \cite{schultz2013}, establishing the cubic identity
\begin{equation} \label{eq:schultz}
\begin{split}
&\biggl(
\sum_{m,n\in\Z} \omega^{m-n} q^{(m+1/3)^2+(m+1/3)(n+1/3)+(n+1/3)^2} x^{m+1/3} y^{n+1/3} \biggr)^3\\
&\quad+\biggl(\sum_{m,n\in\Z} q^{m^2+mn+n^2}x^m y^n \biggr)^3 = 
\biggl(\sum_{m,n\in\Z} \omega^{m-n} q^{m^2+mn+n^2}x^my^n \biggr)^3\\
&\qquad\qquad+\biggl( \sum_{m,n\in\Z} q^{(m+1/3)^2+(m+1/3)(n+1/3)+(n+1/3)^2} x^{m+1/3} y^{n+1/3} \biggr)^3,
\end{split}
\end{equation}
where $q=e^{2\pi i\tau}$, $\tau\in\Hsp$, $x,y\in\C\ssm\{0\}$ and $\omega:=e^{2\pi i/3}$, via analytic methods involving elliptic functions and $q$-series expansions.
Here we provide a purely geometric proof of \eqref{eq:schultz}.
Our method leverages a general structural property of quadratic forms under tensor products:
given an automorphism $\gamma$ of a positive definite quadratic form $A$, we use complementary orthogonal projections to construct a rational automorphism of $I_n\otimes A$ that fixes the diagonal subspace pointwise.
Using elementary lattice decompositions and finite character sums, we then obtain transformation formulas for the associated theta functions.
Our main result is given in Section \ref{sec:5}.
It provides a unified geometric approach to Schultz's identity and the related quadratic identities considered in \cite{chan2025}.

Our work was developed independently of the recent preprint \cite{chen2026}, which establishes a general theta-product formula for rationally isometric quadratic spaces and derives, among other applications, a new proof of Schultz's identity.
In their proof, the authors use a particular rational isometry of a six-dimensional quadratic space.
Our approach is different in that we provide a simple, systematic construction of such isometries from automorphisms of the underlying quadratic form.

\section{Space-Splitting Automorphisms}

Throughout this note let $A\in M_r(\Z)$ be a symmetric positive definite matrix, representing the quadratic form $x\mapsto \frac12x^\tr A x$.
We consider
\[\begin{split}
\Ogroup_A(\Q)&:= \{\gamma \in \operatorname{GL}_r(\Q) \mid \gamma^\tr A \gamma =A\},\\
\Ogroup_A(\Z)&:= \{\gamma \in \operatorname{GL}_r(\Z) \mid \gamma^\tr A \gamma =A\}= \Ogroup_A(\Q)\cap \operatorname{GL}_r(\Z).
\end{split}\]
The elements from these groups preserve the corresponding quadratic form.
Further, the elements from $\Ogroup_A(\Z)$ also preserve the lattice $\Z^r$.

\begin{lemma} \label{lem:projection_props}
Let $n\in\N$, $e:=(1,\ldots,1)\in \Z^n$, $E_n:= \frac1n ee^\tr \in M_n(\Q)$ and $R_n:=I_n-E_n\in M_n(\Q)$.
Then $E_n^2 = E_n = E_n^\tr$, $E_nR_ n=R_nE_n=0$, $R_n^2 = R_n= R_n^\tr$ and $E_n+R_n=I_n$.
Further, $E_ne=e$ and $R_ne=0$.
\end{lemma}

\begin{proof}
Trivial.
\end{proof}

\begin{lemma}\label{lem:construction}
For $n\geq 2$ and $\gamma\in \Ogroup_A(\Q)$ let
\[g := E_n\otimes I_r + R_n \otimes \gamma= I_n\otimes \gamma + E_n\otimes (I_r-\gamma)\in M_{nr}(\Q).\]
Then
\begin{enumerate}
\item $g\in \Ogroup_{I_n\otimes A}(\Q)$,
\item the map $\gamma \mapsto E_n\otimes I_r + R_n \otimes \gamma$ is an injective homomorphism,
\item $g (e \otimes x)=e\otimes x=g^\tr (e\otimes x)$ for all $x\in\C^r$.
\end{enumerate}
\end{lemma}

\begin{proof}
Using the algebraic rules of Kronecker products, Lemma~\ref{lem:projection_props} and $\gamma^\tr A\gamma=A$ these are easy calculations:
we have
\[\begin{split}
g^\tr (I_n \otimes A) g &= (E_n^\tr \otimes I_r + R_n^\tr \otimes \gamma^\tr ) (I_n \otimes A) ( E_n \otimes I_r+ R_n \otimes \gamma) \\
&= E_n^\tr E_n \otimes A + E_n^\tr R_n \otimes A\gamma +  R_n^\tr E_n \otimes \gamma^\tr A + R_n^\tr R_n \otimes \gamma^\tr A \gamma \\
&= E_n \otimes A + R_n \otimes A = I_n \otimes A,
\end{split}\]
so $g\in \Ogroup_{I_n\otimes A}(\Q)$.
Similarly
\[\begin{split}
(E_n \otimes I_r + R_n \otimes \gamma_1) (E_n \otimes I_r+ R_n\otimes \gamma_2)= E_n\otimes I_r+ R_n\otimes \gamma_1\gamma_2.
\end{split}\]
From $E_n \otimes I_r + R_n \otimes \gamma_1=E_n \otimes I_r + R_n \otimes \gamma_2$ we get $R_n\otimes(\gamma_1-\gamma_2)=0$ and $\gamma_1=\gamma_2$.
Finally
\[ \begin{split}
g(e\otimes x)&= (E_n\otimes I_r+R_n\otimes \gamma)(e\otimes x) = E_ne\otimes x + R_ne\otimes \gamma x = e\otimes x,\\
g^\tr(e\otimes x)&= (E_n^\tr \otimes I_r+R_n^\tr \otimes \gamma^\tr)(e\otimes x) =E_n^\tr e\otimes x + R_n^\tr e\otimes \gamma^\tr x= e\otimes x.\qedhere
\end{split}\]
\end{proof}

\section{Shifted Lattices}

From now on let $\gamma\in \Ogroup_A(\Z)$.
Then $\gamma$ preserves the lattice $\Z^r$, but the corresponding matrix $g$ does not necessarily preserve $\Z^{nr}$.

\begin{definition}
(a) We set $G:= \tfrac1n(I_r-\gamma)\in M_r(\Q)$.\\
(b) We view an element $m$ of $\Z^{nr}=(\Z^r)^n$ as a vector of $n$ vectors $m_1,\ldots,m_n\in\Z^r$ and set
\[\langle e,m \rangle :=m_1+\ldots+m_n \in \Z^r.\]
\end{definition}

\begin{lemma}\label{lem:angle}
We have $\langle e,gx\rangle=\langle e,x\rangle$ for all $x\in(\R^r)^n$.
\end{lemma}

\begin{proof}
We write $g$ as $g= I_n\otimes \gamma +ee^\tr \otimes G$, so
\[gx=(I_n\otimes \gamma)x + (ee^\tr \otimes G)x= (I_n\otimes \gamma)x+ e\otimes G\langle e,x\rangle,\]
from which we find
\[ \langle e,gx\rangle=\langle e,(I_n\otimes \gamma)x\rangle +\langle e,e\otimes G\langle e,x\rangle\rangle=(\gamma+n G)\langle e,x\rangle=\langle e,x\rangle.\qedhere\]
\end{proof} 

Now consider the lattice $L := G\Z^r+\Z^r$.
It contains $\Z^r$ as a sublattice and has $\frac1n\Z^r$ as an overlattice, so the index of $\Z^r$ in $L$ divides $n^r$.
Let $C\subset G\Z^r$ be a complete set of coset representatives of $\Z^r$ in $L$. 
For $a,d\in G\Z^r$ consider the shifted lattice
\[ W_{a,d} := \{ m\in e\otimes a + (\Z^r)^n \mid G\langle e,m \rangle \in d + \Z^r\}.
\]
The set $W_{a,d}$ depends only on the cosets of $a$ and $d$ modulo $\Z^r$.

\begin{lemma}\label{lem:cosets}
For all $a\in G\Z^r$ we have
\[ e\otimes a+(\Z^r)^n=\bigsqcup_{d\in C} W_{a,d}.\]
\end{lemma}

\begin{proof}
Let $m\in e\otimes a +(\Z^r)^n$.
We have $\langle e, e\otimes a\rangle=n a\in \Z^r$, so we get $\langle e,m\rangle\in \Z^r$ and $G\langle e,m\rangle \in L$.
Hence for every $m\in e\otimes a +(\Z^r)^n$ there is a unique $d\in C$ with $G\langle e,m\rangle\in d+\Z^r$.
This shows $e\otimes a+(\Z^r)^n\subset\bigsqcup_{d\in C} W_{a,d}$.
The inclusion in the other direction is trivial.
\end{proof}

\begin{lemma}\label{lem:inv}
For all $a,d\in G\Z^r$ we have $g W_{a,d} = W_{\gamma a+d,d}$.
\end{lemma}

\begin{proof}
Let $m\in W_{a,d}$.
Then we have $(I_n\otimes \gamma)m\in e\otimes (\gamma a)+ (\Z^r)^n$ and $e\otimes G\langle e,m\rangle\in e\otimes d +(\Z^r)^n$, so we find $gm=(I_n\otimes \gamma)m+e\otimes G\langle e,m\rangle\in e\otimes (\gamma a+d) +(\Z^r)^n$. 
Further, we have $G\langle e,gm\rangle =G\langle e,m\rangle \in d+\Z^r$, so we get $gm\in W_{\gamma a+d,d}$ and $gW_{a,d}\subset W_{\gamma a+d,d}$.
Similarly we have $g^{-1}W_{\gamma a+d,d}\subset W_{a,d}$:
let $m\in W_{\gamma a+d,d}$.
With $g^{-1}=(I_n\otimes \gamma^{-1})-ee^\tr \otimes (\gamma^{-1}G)$ we find
\[ \begin{split}
g^{-1} m &= (I_n\otimes \gamma^{-1})m - e\otimes (\gamma^{-1} G\langle e,m\rangle)\\
&\in e\otimes \bigl(\gamma^{-1}(\gamma a+d)-\gamma^{-1}d\bigr)+(\Z^r)^n=e\otimes a +(\Z^r)^n.
\end{split}\]
Further, we have $G\langle e,g^{-1}m\rangle=G\langle e,m\rangle \in d+\Z^r$, so $g^{-1}m\in W_{a,d}$.
This shows $g^{-1}W_{\gamma a+d,d}\subset W_{a,d}$, $W_{\gamma a+d,d}\subset gW_{a,d}$ and $g W_{a,d} = W_{\gamma a+d,d}$.
\end{proof}

\section{Indicator functions}

Similar to the definition of $L$ and $C$ from the previous section, we let $L^\tr:= G^\tr\Z^r+\Z^r$ and $C^\tr\subset G^\tr \Z^r$ be a complete set of coset representatives of $\Z^r$ in $L^\tr$.

\begin{definition}
For $a=Gu\in G\Z^r$ and $b=G^\tr v\in G^\tr \Z^r$, with $u,v\in\Z^r$, we define
\[[a,b]:= u^\tr G^\tr v \smod\Z.\]
\end{definition}

\begin{remark}
This definition of $[a,b]$ is independent of the choice of preimages $u,v\in\Z^r$:
if we have $Gu=Gu'$ and $G^\tr v=G^\tr v'$, then $(u-u')^\tr G^\tr v=0$ and $u^\tr G^\tr (v-v')=0$, so $u^\tr G^\tr v=u'^\tr G^\tr v$ and $u^\tr G^\tr v=u^\tr G^\tr v'$.
\end{remark}

\begin{definition}
We extend the definition of the pairing $[\cdot,\cdot]$ to $L\times L^\tr$ by defining $[m,m']:=[a,b]$, where $a\in C$ is the unique representative of the coset $m+\Z^r$ in $L$ and $b\in C^\tr$ the unique representative of the coset $m'+\Z^r$ in $L^\tr$.
\end{definition}

\begin{remark}
(a) If $a=Gu,a'=Gu'\in G\Z^r\subset L$ represent the same coset of $\Z^r$ in $L$, then we have $a-a'\in\Z^r$, so $[a,b]-[a',b]\equiv(u-u')^\tr G^t v = (a-a')^\tr v\equiv 0\smod\Z$ and $[a,b]=[a',b]$.
Similarly, the pairing is unchanged when $b$ is replaced by an element of the same coset of $\Z^r$ in $L^\tr$.\\
(b) From the definition and the independence of representatives we directly get that this pairing is bilinear.
\end{remark}

\begin{lemma}\label{lem:decomp}
For $m\in L$ we have
\[\frac1{| C^\tr|} \sum_{b\in C^\tr} e^{2\pi i[m,b]}= \begin{cases} 1, &\text{if}\ m\in \Z^r,\\
0, & \text{otherwise}.\end{cases}\]
\end{lemma}

\begin{proof}
If $m\in\Z^r$, then $e^{2\pi i[m,b]}=e^{2\pi i[0,b]}=1$ and
\[\frac1{| C^\tr|}  \sum_{b\in C^\tr} e^{2\pi i[m,b]} = \frac1{| C^\tr|} \sum_{b\in C^\tr} 1= 1.\]
Otherwise, let $a\in C$ represent the coset $m+\Z^r$. 
Since $a\notin\Z^r$, there exists $v'\in\Z^r$ such that $a^\tr v'\notin\Z$.
Hence $b':=G^\tr v'$ satisfies $e^{2\pi i[m,b']}\neq 1$.
If $b$ runs through a complete set of coset representatives of $\Z^r$ in $L^\tr$, then so does $b+ b'$.
Since $e^{2\pi i[m,b]}$ depends only on the coset of $b$ modulo $\Z^r$, we obtain
\[ \sum_{b\in C^\tr} e^{2\pi i[m,b]}=\sum_{b\in C^\tr} e^{2\pi i[m,b+b']}=e^{2\pi i[m,b']} \sum_{b\in C^\tr}e^{2\pi i[m,b]}\quad \text{and} \quad \sum_{b\in C^\tr} e^{2\pi i[m,b]}=0.\qedhere\]
\end{proof}

\begin{lemma}\label{lem:split}
Let $a,d\in C$.
Assuming absolute convergence we have
\[ \sum_{\mu\in W_{a,d}} f(\mu)=\frac1{| C^\tr|}  \sum_{b\in C^\tr} e^{-2\pi i[d,b]}\sum_{\mu\in e\otimes a+(\Z^r)^n} e^{2\pi i \mu^\tr (e\otimes b)} f(\mu).\]
\end{lemma}

\begin{proof}
For $\mu\in e\otimes a+(\Z^r)^n$, we have $\mu\in W_{a,d}$ if and only if $G\langle e,\mu\rangle-d\in\Z^r$.
Note that $G\langle e,\mu\rangle-d\in L$.
Using Lemma \ref{lem:decomp} we thus get
\[ \sum_{\mu\in W_{a,d}} f(\mu) = \frac1{|C^\tr|} \sum_{b\in C^\tr} \sum_{\mu\in e\otimes a +(\Z^r)^n} e^{2\pi i[G\langle e,\mu\rangle -d,b]}f(\mu).\]
The result then follows from
\[e^{2\pi i[G\langle e,\mu\rangle -d,b]}= e^{-2\pi i[d,b]+2\pi i\langle e,\mu\rangle^\tr G^\tr v} = e^{-2\pi i[d,b]+2\pi i\langle e,\mu\rangle^\tr b}= e^{-2\pi i[d,b]+2\pi i\mu^\tr(e\otimes b)},\]
where we used $b=G^\tr v$, with $v\in\Z^r$.
\end{proof}

\section{A General Result for Theta}\label{sec:5}

\begin{theorem}\label{theo:main}
For $a\in C$ and $b\in C^\tr$ let
\[ \Theta_{a,b}(z):= \sum_{\lambda \in e\otimes a +(\Z^r)^n} q^{\frac12 \lambda^\tr (I_n\otimes A) \lambda}\, e^{2\pi i \lambda^\tr (z+ e\otimes b)}\qquad (z\in(\C^r)^n).\]
Then
\[\begin{split}
\Theta_{a,b}\bigl(g^\tr z\bigr) &= \frac1{|C^\tr|}\sum_{a'\in C} \sum_{b'\in C^\tr} e^{-2\pi i[a'-\gamma a,b'-b]}\,\Theta_{a',b'}(z),\\
\Theta_{a,b}\bigl(g^{-\tr} z\bigr) &= \frac1{|C^\tr|}\sum_{a'\in C} \sum_{b'\in C^\tr} e^{-2\pi i[a-\gamma a',b'-b]}\,\Theta_{a',b'}(z).
\end{split}\]
\end{theorem}

\begin{remark}
(a) If $a,a'\in G\Z^r\subset L$ are in the same coset of $\Z^r$ in $L$, then $a'\in a+\Z^r$, so $e\otimes a' +(\Z^r)^n=e\otimes a+(\Z^r)^n$.
Hence $\Theta_{a,b}$ doesn't depend on the choice of representative $a$.
It also doesn't depend on the choice of representative $b$:
if $b,b'\in G^\tr\Z^r\subset L^\tr$ are in the same coset of $\Z^r$ in $L^\tr$, then $b-b'\in\Z^r$, so for all $\lambda\in e\otimes a +(\Z^r)^n$ we have
\[ e^{2\pi i\lambda^\tr(e\otimes (b-b'))}= e^{2\pi i(e\otimes a)^\tr(e\otimes (b-b'))}=e^{2\pi ina^\tr (b-b')}=1,\]
where we used that $na\in\Z^r$.
Hence we find $e^{2\pi i\lambda^\tr(e\otimes b)}=e^{2\pi i\lambda^\tr(e\otimes b')}$ and $\Theta_{a,b}=\Theta_{a,b'}$.\\
(b) For $x\in\C^r$ we have
\[ \Theta_{a,b}(e\otimes x) = \theta_{a,b}(x)^n,\qquad \text{with}\qquad \theta_{a,b}(x):= \sum_{\lambda\in a+\Z^r} q^{\frac12\lambda^\tr A \lambda}\,  e^{2\pi i\lambda^t(x+b)}.\]
\end{remark}

\begin{proof}[Proof of Theorem \ref{theo:main}]
Let \[ \widetilde\Theta_{a,d,b}(z):= \sum_{\lambda \in W_{a,d}} q^{\frac12 \lambda^\tr (I_n\otimes A) \lambda}\, e^{2\pi i \lambda^\tr (z+ e\otimes b)}.\]
Using Lemma \ref{lem:cosets} and Lemma \ref{lem:split} we get
\begin{align}
\Theta_{a,b}(z)&=\sum_{d\in C} \widetilde\Theta_{a,d,b}(z),\label{pro1}\\
\widetilde\Theta_{a,d,b}(z)&= \frac1{|C^\tr|} \sum_{b'\in C^\tr}e^{-2\pi i[d,b']}\, \Theta_{a,b+b'}(z).\label{pro2}
\end{align}
We now consider $\widetilde\Theta_{a,d,b}(g^\tr z)$ and substitute $\lambda=g^{-1}\mu$ in the definition.
From Lemma \ref{lem:construction} we get
\[ (g^{-1})^\tr (I_n\otimes A) g^{-1}=I_n\otimes A,\qquad g(e\otimes a)=e\otimes a\qquad \text{and} \qquad (g^{-1})^\tr (e\otimes b)=e\otimes b.\]
Further, we have $gW_{a,d}=W_{\gamma a+d,d}$ (Lemma \ref{lem:inv}), so
\begin{equation}\label{pro3}
\widetilde\Theta_{a,d,b}(g^\tr z) = \widetilde\Theta_{\gamma a+d,d,b}(z).
\end{equation}
Using \eqref{pro1}, \eqref{pro3} and then \eqref{pro2} thus gives
\[ \begin{split}
\Theta_{a,b}(g^\tr z)&=\sum_{d\in C} \widetilde\Theta_{a,d,b}(g^\tr z)=\sum_{d\in C} \widetilde\Theta_{\gamma a+d,d,b}(z)\\
&= \frac1{|C^\tr|}\sum_{d\in C} \sum_{b'\in C^\tr}e^{-2\pi i[d,b']}\, \Theta_{\gamma a+d,b+b'}(z)= \frac1{|C^\tr|}\sum_{a'\in C} \sum_{b'\in C^\tr} e^{-2\pi i[a'-\gamma a,b'-b]}\, \Theta_{a',b'}(z),
\end{split}\]
where we replaced $b'$ by $b'-b$ and substituted $d=a'-\gamma a$ in the last step.
From \eqref{pro3} we also get
\[\widetilde\Theta_{a,d,b}(g^{-\tr}z)=\widetilde\Theta_{\gamma^{-1}(a-d),d,b}(z),\]
which similarly gives
\[ \begin{split}
\Theta_{a,b}&(g^{-\tr}z)=\sum_{d\in C} \widetilde\Theta_{a,d,b}(g^{-\tr} z)=\sum_{d\in C} \widetilde\Theta_{\gamma^{-1}(a-d),d,b}(z)\\
&= \frac1{|C^\tr|}\sum_{d\in C} \sum_{b'\in C^\tr}e^{-2\pi i[d,b']}\, \Theta_{\gamma^{-1} (a-d),b+b'}(z)= \frac1{|C^\tr|}\sum_{a'\in C} \sum_{b'\in C^\tr} e^{-2\pi i[a-\gamma a',b'-b]}\, \Theta_{a',b'}(z),
\end{split}\]
where we replaced $b'$ by $b'-b$ and substituted $d=a-\gamma a'$ in the last step.
\end{proof}

\section{Schultz's Cubic Theta Identity as an Example}\label{sec:schultz}

Let $A=\left(\begin{smallmatrix} 2&1\\1&2\end{smallmatrix}\right)$.
The group $\Ogroup_A(\Z)$ has twelve elements and is isomorphic to the dihedral group $D_6$.
We consider $\gamma= \left(\begin{smallmatrix} -1&-1\\1&0\end{smallmatrix}\right)\in \Ogroup_A(\Z)$ of order three and $n=3$.
We have $G=\frac13\left(\begin{smallmatrix} 2&1\\-1&1\end{smallmatrix}\right)$ and $L=\{m\in \frac13\Z^2 \mid m_1-m_2\in\Z\}$.
A complete set $C\subset G\Z^2$ of coset representatives of $\Z^2$ in $L$ is $C=\{ 0,\frac13u,\frac23u\}$, with $u:=(1,1)^\tr$.
Similarly, $G^\tr=\frac13\left(\begin{smallmatrix} 2&-1\\1&1\end{smallmatrix}\right)$, $L^\tr=\{m\in \frac13\Z^2 \mid m_1+m_2\in\Z\}$ and $C^\tr=\{ 0,\frac13v,\frac23v\}$, with $v:= (1,-1)^\tr$.
Theorem \ref{theo:main} gives ($\alpha,\beta\in\{0,\frac13,\frac23\}$)
\[\begin{split}
\Theta_{\alpha u,\beta v}(g^\tr z) &= \frac13 \sum_{\alpha',\beta'\in\{0,\frac13,\frac23\}}\omega^{9(\alpha'-\alpha)(\beta'-\beta)}\,\Theta_{\alpha' u,\beta' v}(z),\\
\Theta_{\alpha u,\beta v}(g^{-\tr} z) &= \frac13 \sum_{\alpha',\beta'\in\{0,\frac13,\frac23\}}\omega^{-9(\alpha'-\alpha)(\beta'-\beta)}\,\Theta_{\alpha' u,\beta' v}(z).
\end{split}\]
Using $\omega^k-\omega^{-k}= i\sqrt 3\,\ileg k3$ (where $\ileg{\cdot}{\cdot}$ is the Legendre symbol) we thus find
\[ \Theta_{\alpha u,\beta v}(g^\tr z)-\Theta_{\alpha u,\beta v}(g^{-\tr} z) =\frac i{\sqrt{3}}\sum_{\alpha',\beta'\in\{0,\frac13,\frac23\}} \leg{9(\alpha'-\alpha)(\beta'-\beta)}3\,\Theta_{\alpha' u,\beta' v}(z).\]
With $z=e\otimes x$ this gives
\[ \sum_{\alpha',\beta'\in\{0,\frac13,\frac23\}} \leg{9(\alpha'-\alpha)(\beta'-\beta)}3\,\theta_{\alpha' u,\beta' v}(x)^3=0.\]
For $\alpha=\beta=\frac23$ we get the following identity, which is equivalent to \eqref{eq:schultz}:
\[ \theta_{0,0}(x)^3-\theta_{0,\frac13v}(x)^3-\theta_{\frac13u,0}(x)^3+\theta_{\frac13u,\frac13v}(x)^3=0\]

\section{Further Examples}

\begin{example}
Let $A=\left(\begin{smallmatrix} 2d&b\\b&2d\end{smallmatrix}\right)$ with $-2d<b<2d$, let $\gamma= \left(\begin{smallmatrix} 0&-1\\-1&0\end{smallmatrix}\right)\in \Ogroup_A(\Z)$ of order two and let $n=2$.
We have $G=G^\tr=\frac12\left(\begin{smallmatrix} 1&1\\1&1\end{smallmatrix}\right)$ and $L=L^\tr=\{m\in \frac12\Z^2 \mid m_1-m_2\in\Z\}$.
Then $C=C^\tr=\{0,\frac12 u\}$, with $u:=(1,1)^\tr$, and so Theorem \ref{theo:main} gives ($\alpha,\beta\in\{0,\frac12\}$)
\[ \Theta_{\alpha u,\beta u}(g^\tr z) = \frac12 \sum_{\alpha',\beta'\in\{0,\frac12\}}(-1)^{4(\alpha'+\alpha)(\beta'-\beta)}\,\Theta_{\alpha' u,\beta' u}(z).\]
With $\alpha=\beta=0$ and $z=e\otimes x$ this gives the following identity, which is equivalent to Corollary 5.1 in \cite{chan2025}:
\[ \theta_{0,0}(x)^2-\theta_{0,\frac12u}(x)^2-\theta_{\frac12u,0}(x)^2+\theta_{\frac12u,\frac12u}(x)^2=0\]

\end{example}

\begin{example}
Let $A=\left(\begin{smallmatrix} 2a&a\\a&2b\end{smallmatrix}\right)$ with $0<a<4b$, let $\gamma= \left(\begin{smallmatrix} -1&-1\\0&1\end{smallmatrix}\right)\in \Ogroup_A(\Z)$ of order two and let $n=2$.
We have $G=\frac12\left(\begin{smallmatrix} 2&1\\0&0\end{smallmatrix}\right)$, $L=\{m\in \frac12\Z^2 \mid m_2\in\Z\}$ and $C=\{0,\frac12u\}$, with $u=(1,0)^\tr$.
Further, $G^\tr=\frac12\left(\begin{smallmatrix} 2&0\\1&0\end{smallmatrix}\right)$, $L^\tr=\{m\in \frac12\Z^2 \mid m_1\in\Z\}$ and $C^\tr=\{0,\frac12v\}$, with $v=(2,1)^\tr$.
Theorem \ref{theo:main} gives ($\alpha,\beta\in\{0,\frac12\}$)
\[ \Theta_{\alpha u,\beta v}(g^\tr z) = \frac12 \sum_{\alpha',\beta'\in\{0,\frac12\}}(-1)^{4(\alpha'+\alpha)(\beta'-\beta)}\,\Theta_{\alpha' u,\beta' v}(z).\]
With $\alpha=\beta=0$ and $z=e\otimes x$ this gives the following identity, which is equivalent to Corollary 5.2 in \cite{chan2025}:
\[ \theta_{0,0}(x)^2-\theta_{0,\frac12v}(x)^2-\theta_{\frac12u,0}(x)^2+\theta_{\frac12u,\frac12v}(x)^2=0\]
\end{example}

\end{document}